\documentclass[a4paper, 11pt,reqno]{amsart}
\usepackage[top = 1in, bottom = 1in, left=1.3in, right=1.3in]{geometry}
\usepackage{setspace} 

\usepackage[utf8]{inputenc}
\usepackage[T1]{fontenc}

\usepackage{amssymb,amsmath,amsthm,graphicx,xcolor,mathtools,mathrsfs}

\usepackage{hyperref}
\hypersetup{colorlinks, linkcolor={red!50!black}, citecolor={green!50!black}, urlcolor={blue!50!black}}

\usepackage{csquotes}

\usepackage[caption=false]{subfig} 

\usepackage{cleveref}
\theoremstyle{plain}
\newtheorem{theorem}{Theorem}[section]
\crefname{theorem}{Theorem}{Theorems}

\newtheorem{proposition}[theorem]{Proposition}
\crefname{proposition}{Proposition}{Propositions}

\newtheorem{corollary}[theorem]{Corollary}
\crefname{corollary}{Corollary}{Corollaries}

\newtheorem{lemma}[theorem]{Lemma}
\crefname{lemma}{Lemma}{Lemmas}

\crefname{conjecture}{Conjecture}{Conjectures}

\crefname{problem}{Problem}{Problem}

\crefname{claim}{Claim}{Claims}

\crefname{observation}{Observation}{Observations}

\crefname{setup}{Setup}{Setups}

\crefname{myth}{Myth}{Myths}

\crefname{fact}{Fact}{Facts}

\crefname{algorithm}{Algorithm}{Algorithms}

\newtheorem{remark}[theorem]{Remark}
\crefname{remark}{Remark}{Remarks}

\crefname{example}{Example}{Examples}

\theoremstyle{definition}
\newtheorem{definition}[theorem]{Definition}
\crefname{definition}{Definition}{Definitions}

\crefname{construction}{Construction}{Constructions}

\crefname{question}{Question}{Questions}

\numberwithin{equation}{section}
\makeatletter
\newcommand\thankssymb[1]{\textsuperscript{\@fnsymbol{#1}}}
\makeatother

\usepackage[shortlabels]{enumitem}
\setlist[enumerate,1]{label={\upshape (\roman*)}}

\usepackage{mathtools}

\allowdisplaybreaks
\DeclareMathOperator{\FP}{FP}

\DeclareMathOperator{\Maj}{Maj}

\usepackage{float}

\author[Aracena]{Julio Aracena}
\author[Astete-Elguin]{Raúl Astete-Elguin\thankssymb{1}}

\address[Aracena]{Departamento de Ingeniería Matemática, Facultad de Ciencias Físicas y Matemáticas, Universidad de Concepción, Chile.}
\email{jaracena@udec.cl}
\address[Astete-Elguin]{Departamento de Ingeniería Informática y Ciencias de la Computación, Facultad de Ingeniería, Universidad de Concepción, Chile.}
\email{
rastete2018@udec.cl}
\thanks{\thankssymb{1} Corresponding author}

\thanks{The research leading to these results was supported by ANID-Chile - Magíster Nacional and Proyecto Basal FB 210005 from ANID and Center for Mathematical Modelling (CMM)}

\title{K-INDEPENDENT BOOLEAN NETWORKS}

\begin{document}

\begin{abstract}
This paper proposes a new parameter for studying Boolean networks: the independence number. We establish that a Boolean network is $k$-independent if, for any set of $k$ variables and any combination of binary values assigned to them, there exists at least one fixed point in the network that takes those values at the given set of $k$ indices. In this context, we define the independence number of a network as the maximum value of $k$ such that the network is $k$-independent. This definition is closely related to widely studied combinatorial designs, such as  ``$k$-strength covering arrays'', also known as Boolean sets with all $k$-projections surjective. Our motivation arises from understanding the relationship between a network's interaction graph and its fixed points, which deepens the classical paradigm of research in this direction by incorporating a particular structure on the set of fixed points, beyond merely observing their quantity. Specifically, among the results of this paper, we highlight a condition on the in-degree of the interaction graph for a network to be $k$-independent, we show that all regulatory networks are at most $n/2$-independent, and we construct $k$-independent networks for all possible $k$ in the case of monotone networks with a complete interaction graph.
\end{abstract}

\maketitle

\section{Introduction}

\subsection{Boolean networks and covering arrays}
A Boolean network (BN) on $n$ variables is a function $f:\{0,1\}^n\to\{0,1\}^n$,
where $f(x) = (f_1(x), \ldots, f_n(x))$ for $x\in\{0,1\}^n$. Each function $f_i:\{0,1\}^n\to\{0,1\}$ is a local activation function of the network.
For $x\in \{0,1\}^n$ we denote by $w_H(x)$ the Hamming weight of $x$. That is, the number of ones of $x$. Additionally we denote $[n] := \{1,\ldots, n\}$. The following are some examples of families of Boolean networks:

\begin{itemize}
    \item \textbf{Linear networks:} Boolean networks where each local activation function is the sum modulo two of some variables.

    \item \textbf{Majority networks:} Networks where each local activation function take the value of the majority of the variables they depend on.

    \item \textbf{AND-OR networks:} Boolean networks in which each local activation function is either a disjunction or a conjunction of the variables they depend on.

    \item \textbf{Monotone networks:} Given $x,y\in\{0,1\}^n$, denote $x\leq y$ if $x_i\leq y_i$ for every $i\in [n]$. A Boolean network $f$ is said to be monotone if it is increasing with respect to the relation $\leq$. Majority networks are a particular case of monotone networks.

    \item \textbf{Regulatory networks:} 
A Boolean function $h:\{0,1\}^n\to\{0,1\}$ is increasing with respect to the variable $i$ if, for every $x\in \{0,1\}^n$, $h(x\, : \, x_i = 0)\leq h(x\, : \, x_i = 1)$ and is said to be decreasing on $i$ if 
for every $x\in \{0,1\}^n$, $h(x\, : \, x_i = 0)\geq h(x\, : \, x_i = 1)$. A Boolean function is {\it unate} if for every $i\in[n]$, $h$ is either increasing or decreasing with respect to the input $i$.
A Regulatory Boolean network is a Boolean network where each local activation function is unate, and since every monotone Boolean function is unate, every monotone Boolean network is also regulatory.

\end{itemize}

In this article, our primary focus will be on linear and monotone networks. In general,  Boolean networks represent  $n$ variables interacting and evolving discretely over time based on a predefined rule. Introduced by Kauffman in 1969 ~\cite{kauffman1969metabolic, kauffman1993origins}, BNs find applications in diverse fields such as social networks ~\cite{green2007emergence}, genetic networks ~\cite{akutsu1999identification}, and biochemical systems ~\cite{helikar2011boolean}.

 In this context, the iteration digraph of a network $f$ over the vertices $\{0,1\}^n$ is defined such that the arcs are of the form $(x,f(x))$ for $x\in\{0,1\}^n$. Each iteration digraph fully represents a Boolean network. However, their utilization becomes impractical due to their large number of nodes. For this reason, associated with any Boolean network $f$, we can define the interaction (or dependency) digraph $G(f)$, with vertices $[n]$ and arcs $(i, j)$ indicating that $f_j$ ``depends'' on variable $i$, i.e., there exists $x\in \{0,1\}^n$ such that
\begin{equation*}
    f_j(x_1,\ldots, x_i = 0, \ldots, x_n)\neq f_j(x_1,\ldots, x_i = 1,\ldots, x_n).
\end{equation*}

It is important to note that $G(f)$ may have loops, i.e., arcs from a vertex to itself. A fixed point of $f$ is a vector $x\in \{0,1\}^n$ such that $f(x) = x$. We will denote the \textbf{set of fixed points}  by $\FP(f) = \{x\in \{0,1\}^n\,:\, f(x) = x\}$. The set of fixed points in a BN is an intriguing subject of study for various reasons. One of them is its significance in applications within biological systems, as they can be interpreted as stable patterns of gene expression. It is also of interest to understand, at a theoretical level, the configurations that lead a Boolean network to stabilize, that is, periodic points ~\cite{veliz2012computation, gadouleau2015fixed}, meaning the states \(x \in \{0,1\}^n\) such that \(f^\ell(x) = x\) for some \(\ell\). Fixed points (case $\ell = 1$) are particularly interesting for inferring information about the activation functions of the network ~\cite{KRUPA2002}. However, most works in this direction study the relationship between the number of fixed points of a Boolean network and the properties of the local activation functions ~\cite{aracena2008maximum, aracena2017number} or of its interaction graph. The information that can be obtained about the architecture of a Boolean network from \textbf{structural} properties of its fixed points has not been thoroughly explored. A first step in this direction is the work carried out in ~\cite{Osorio}, where the VC dimension in Boolean networks is defined in terms of their fixed points.

Given $x\in \{0,1\}^n$ and a set of indices $I = \{i_1,\ldots, i_k\}\subseteq [n]$ we denote $x_I = (x_{i_1},\ldots, x_{i_k})$. A covering array of strength $k$ is defined as a set of Boolean vectors from $\{0,1\}^n$ such that for every subset $I$ of $k$ indices, and for every $a = (a_1,\ldots, a_k)\in \{0,1\}^k$, there exists a vector $x$ in the set such that $x_I = a$.  In addition, we denote $CA(m,n;k)$ as the set of all covering arrays with $m$ vectors of size $n$ and strength $k$. When we do not need to refer to the number of rows, we simply denote it by $CA(n;k)$. For example, the following is an element of $CA(5,4;2):$

\begin{align*}
    B=\begin{matrix}
0&0&0&0\\1&0&1&1\\0&1&1&1\\1&1&0&1\\1&1&1&0
\end{matrix}.
\end{align*}
 One of the main challenges of covering arrays is to determine those with the least possible number of elements while maintaining strength. $CAN(n;k)$ the minimum number of rows of a matrix in $CA(m,n;k)$. It is worth mentioning that determining $CAN(n;k)$ for arbitrary values of $n$ and $k$ remains an open problem; we can see some of the known values in Table \ref{table:1}. Various efforts have been made to find approximations to this minimum. However, the case of $k = 2$ is the only non-trivial case that has been completely solved ~\cite{kleitman1973families, katona1973two}.

\begin{table}[h!]

\centering
\begin{tabular}{||c| c c c c c c||} 
 \hline
 $s\backslash t$ & 1 & 2 & 3&4&5&6 \\ [0.5ex] 
 \hline\hline
 0 & 2 & 4 & 8 &16 & 32 &64\\
 1 & 2 & 4 & 8 &12 & 32 &64\\
 2 & 2 &  5& 10 & 21 & 42 &85\\
 3 & 2 &  6& 12 & 24 & 48-52 & 96-108\\
 4 & 2 &  6& 12 & 24 & 48-54  & 96-116\\
 5 & 2 &  6& 12 &24 & 48-56 &96-118\\
 6 & 2 & 6& 12 &24 & 48-64 &96-128\\
 7 & 2 & 6 & 12 &24 & 48-64 &96-128\\
 8 & 2 & 6 & 12 & 24& 48-64 &96-128\\
 9 & 2  & 7 & 15 &30-32 & 60-64 & 120-128\\
 10 & 2 & 7 & 15-16 & 30-35 & 60-79& 120-179\\ [1ex] 
 \hline
\end{tabular}
\caption{Some known values of $CAN(s+t; t)$ ~\cite{LJKRLYKRFM2011}.}

\label{table:1}

\end{table}

Considering the preceding discussion, it becomes pertinent to investigate the implications, in terms of the interaction graph of a Boolean network, when its fixed points constitute a covering array of strength $k$. Consequently, we introduce the concept of $k$-independence for a Boolean network on $n$ variables $f$, wherein we define it as possessing fixed points that form an element of $CA(n;k)$. Moreover, we denote by $i(f)$ the maximum $k$ such that $f$ is $k$-independent, and extend this notion to graphs, stating that a graph $G$ on $n$ vertices is $k$-admissible if there exists a $k$-independent Boolean network whose interaction graph is isomorphic to $G$.

It is pertinent to ask why we study the case where fixed points form a covering array. The first reason is 
because it is a particular case of sets that have VC-dimension equals $k$. We believe it could be a significant step towards understanding the structure of fixed points against the structure of the interaction graph. Additionally, while this work introduces a previously unstudied family of Boolean networks, the study in ~\cite{KRUPA2002} addresses an inference problem in networks using covering arrays, referred to as universal matrices. There is also an applied motivation: a network of individuals expressing binary opinions can be modeled by a $k$-independent Boolean network. In such a scenario, any group of $k$ individuals can express any opinion in a stable state, providing a degree of ``independence'' in their opinions. Ultimately, this exploration not only enhances our understanding of Boolean networks but also opens new avenues for investigating their structural properties beyond the traditional focus on the number of fixed points.

\subsection{Our contribution}
As previously mentioned, this work focuses on the concepts of covering arrays and Boolean networks. Our aim is to delve deeper into the fixed points of a Boolean network, examining not only their quantity but also the specific structure of a covering array.

Our work begins by showing the existence of Boolean networks on $n$ variables and $i(f) = k$, for any $1\leq k\leq n$. However, the presented construction requires a complete interaction graph without loops and the network is not monotone.  We present necessary conditions for the existence of a $k$-independent Boolean network in terms of its local activation functions, the number of fixed points, and the properties of its interaction graph. We then show some families of graphs that are $k$-admissible for different values of $k$. In Section \ref{construcciones}, we present general constructions of networks with $i(f) = k$, representing various scenarios for the parameters $m$, $n$, and $k$ of covering arrays in $CA(m,n;k)$. Nevertheless,, these constructions do not explicitly demonstrate the existence of monotone networks with $i(f) = k$. Finally, we address this question in Section \ref{seccionmonotonas}, where we present an existence result that utilizes Steiner systems to construct the local activation functions of a monotone network with $i(f) = k$ on the complete graph without loops.

\section{Results}

\subsection{General results}
In this section, we establish the basic results on the $k$-admissibility of graphs and the existence of Boolean networks with $i(f) = k$. To do this, first, we will review some classical results from the literature concerning fixed points of Boolean networks. A  significant motivation in this area is to answer the question: What can we infer about the fixed points of $f$ based on $G(f)$, and vice versa? The results we present initially compare the number of fixed points of $f$ with properties of $G(f)$. Perhaps one of the most referenced result in this field is the feedback bound. 

Let us recall that, given a directed graph $G = (V, A)$, we define a set $S\subseteq V$ as a feedback vertex set if the subgraph $G[V\setminus S]$ is acyclic. Furthermore, we introduce the transversal number of $G$, denoted by $\tau(G)$, as the minimum cardinality of a feedback vertex set for $G$.

\begin{theorem}[Feeback bound ~\cite{aracena2008maximum}]
        For any Boolean Network $f$ we have:
    \begin{align*}
        |\FP(f)|\leq 2^{\tau(G(f))}
    \end{align*}
\end{theorem}

This result establishes a necessary condition for the $k$-admissibility of graphs. Specifically, for a graph $G$ to be $k$-admissible, it must be the interaction graph of a Boolean network, where the fixed points form a covering array of strength $k$. This requires having at least $2^k$ fixed points.  Moreover, we stipulate that
\begin{align*}
    CAN(n;k) \leq 2^{\tau(G)} \iff \tau(G)\geq \log CAN(n;k) 
\end{align*}

It is important to note that for some values of $n$ and $k$, as seen in Table \ref{table:1}, $$\log CAN(n;k)>k,$$ and therefore in such situations, $k$-admissible graphs require $\tau(G)>k$. For example,  consider a complete bipartite graph $K_{n,2}$. In this case, $\tau(K_{n,2}) = 2$. Then, the feedback bound allows us to establish that for any Boolean network $f$ with interaction graph $K_{n,2}$, $|\FP(f)|\leq 2^{2} = 4
    $. Later, as we have seen in Table \ref{table:1}, for all $n\geq 4$ we have $CAN(n; 2)>4$, we can conclude that for $n\geq 4$, $K_{n,2}$ \textbf{is not } $k$-admissible for any $1<k\leq n$.

Hereafter, we address the problem of the existence of Boolean networks $f: \{0,1\}^n \to \{0,1\}^n$ with $i(f) = k$, for any $1 \leq k \leq n-1$. As we will see, the architecture that allows $k$-independence for any $k$ turns out to be the complete graph on $n$ vertices without loops. This is a reasonable candidate, as it is a graph with a transversal number of $n-1$.

\begin{proposition}\label{completegraph}
    Let $G = K_n$ be the complete graph without loops. Then $G$ is $(n-1)$-admissible. Moreover, for every $1\leq k \leq n-1$, there exists a Boolean network $f$ such that $G(f) = K_n$ and $i(f) = k$. 
\end{proposition}
\begin{proof}
    Assuming linear functions in every node, we can compute that the set of fixed points is the set of every vector in $\{0,1\}^n$ with an even number of ones. This is a known covering array of strength $n-1$ (see, e.g. ~\cite{LJKRLYKRFM2011}). 

Consider $1 \leq k < n-1$, and let 
\begin{align*}
    S_k&:=\{x\in\{0,1\}^n\,:\,w_H(x)= j\leq k+1\,\,\text{and}\,\,  j = 0\mod 2\},
    \\
    T_k &:=\{x\in\{0,1\}^n\,:\,w_H(x)= j\leq k+1\,\,\text{and}\,\, j = 1\mod 2\}.
\end{align*}
We claim that if $k$ is even, $S_k \in CA(n; k) \setminus CA(n; k+1)$, and if $k$ is odd, $T_k \in CA(n; k) \setminus CA(n; k+1)$. Additionally, there exist Boolean networks $f, g : \{0,1\}^n \to \{0,1\}^n$ such that $\FP(f) = S_k$ and $\FP(g) = T_k$. We will prove the case for even $k$; the proof for odd $k$ is analogous.

Let $I = \{i_1, \ldots, i_k\} \subseteq [n]$ and $a = (a_1, \ldots, a_k) \in \{0,1\}^k$. Clearly, $a$ has at most $k$ ones. If $a$ has an even number of ones, consider $x \in \{0,1\}^n$ such that $x_I = a$ and $x_i = 0$ for every $i \not\in I$. Then $x \in S_k$. Now suppose $a$ has an odd number of ones. Consider $x \in \{0,1\}^n$ such that $x_I = a$. Choose $j \in [n] \setminus I$ and let $x_j = 1$, while for every $i \not\in I \cup \{j\}$, $x_i = 0$. Therefore, $x$ has at most $k+1$ ones, and an even number of them, i.e., $x \in S_k$. Thus, $S_k \in CA(n; k)$. If $k$ is even, then $k+1$ is odd. For every $I = \{i_1, \ldots, i_k, i_{k+1}\} \subseteq [n]$, there is no $x \in S_k$ such that $x_I = \vec{1}$. Therefore, $S_k \in CA(n; k) \setminus CA(n; k+1)$

Now define $f:\{0,1\}^n\to\{0,1\}^n$ such that for every $x = (x_1,\ldots, x_n)\in\{0,1\}^n$, $f_i(x) = 1$ iff $w_H(x\setminus x_i)\leq k$ and $w_H(x\setminus x_i)$ is odd. Here we denote $x\setminus x_i 
 := (x_1,\ldots, x_{i-1},x_{i+1},\ldots, x_n)$ and recall that $w_H(x)$ denotes the amount of ones of $x$. Then, it is easy to see that $G(f) = K_n$ and $\FP(f) = S_k$. As a final remark, for the case where $k$ is odd, we define $g:\{0,1\}^n\to\{0,1\}^n$ such that $g_i(x) = 1$ if and only if $w_H(x\setminus x_i) \leq k$ and $w_H(x\setminus x_i)$ is even.
\end{proof}

\begin{remark}\label{no son monotonas}
    The Boolean networks constructed in the previous proposition are non monotone. Indeed, for $k$ even, let $f$ be the network constructed such that $\FP(f) = S_k$. Let $x\in \{0,1\}^n$ such that $w_H(x) = k+1$, and let $y\in\{0,1\}^n$ such that $x\leq y$. We observe an index $i\in [n]$ such that $x_i = 1$. Since $x\leq y$, we have $y_i = 1$, and $w_H(y)\geq k+2$. Therefore, $f_i(y) = 0$, as $w_H(y\setminus y_i)\geq k+1$. This implies that $f(x) = x$, and hence, $f(x) \not\leq f(y)$.
\end{remark}

As we can see in Fig \ref{fig:figure2},  $k$-admissible graphs, with $k \geq 2$, are not necessarily complete, but it is true that they tend to become denser for larger values of $k$. In fact, to prove this, let us first consider the following definition.

\begin{figure*}[tbhp] 
\centering
\subfloat[3-admissible graph with linear interaction]{\includegraphics[width=0.6\columnwidth]{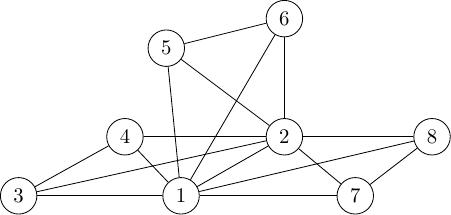}}
\hfill
\subfloat[2-admissible with majority]{\includegraphics[width=0.3\columnwidth]{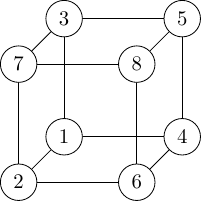}
}
\caption{Examples of $k$-admissible non-complete graphs with $k>1$.}
\label{fig:figure2}
\end{figure*}

\begin{definition}[See e.g. ~\cite{kadelka2023collectively}]
We say that $h:\{0,1\}^n\to \{0,1\}$ is  $k$-set canalizing if there exists a set $I = \{i_1,\ldots, i_k\}\subseteq \{1,\ldots, n\}$ and values $a_1,\ldots, a_k, b\in\{0,1\}$ such that
    \begin{align*}
    \forall x\in\{0,1\}^n, \, x_I = (a_1,\ldots, a_k) \implies h(x) = b
\end{align*}
In this context, we say that the input $a_1,\ldots, a_k$ canalizes $h$ to $b$. Moreover, we denote by $IC(h)$ the minimum $k$ such that $h$ is $k$-set canalizing. 
\end{definition}

It is easy to see that $h$ is $k$-set canalizing if and only if the minimum number of literals in a clause of a DNF-formula (or CNF-formula) of $h$ is $k$. The following are examples of $k$-set canalizing functions for different values of $k$:
\begin{itemize}
\item The AND function $g:\{0,1\}^n\to\{0,1\}$, defined as $g(x_1,\ldots, x_n) = \bigwedge_{i = 1}^n x_i$, is $1$-set canalizing. It canalizes to zero whenever any variable takes the value zero. Similarly, disjunctions are $1$-set canalizing, canalizing to one when any variable takes the value one.


\item The majority function $\Maj:\{0,1\}^n\to\{0,1\}$, defined as 
\begin{equation*}
    \Maj(x_1,\ldots, x_n) = 1\iff w_H(x)\geq \lceil n/2\rceil
\end{equation*}
is such that $IC(\Maj) = \lceil n/2\rceil$.
\end{itemize}

The previous concept allows us to state the following necessary condition for the $k$-independence of a Boolean network.
\begin{theorem}\label{teoremanecesario}
    Let $f = (f_1, \ldots, f_n)$ be a $k$-independent Boolean network such that $G(f)$ has no loops, then for all $i$, $IC(f_i)\geq k$.
\end{theorem}
\begin{proof}
     By contradiction, assume that $f$ is $k$-independent, and that there exists a local activation function $f_i$ that canalizes into $\tilde I = \{i_1, \ldots, i_\ell\} \subseteq N^-(i)$ with $\ell < k$, on inputs $a = (a_1, \ldots, a_\ell) \in \{0,1\}^\ell$ to the value $b \in \{0,1\}$. Since there are no loops, we may assume that $i \notin \tilde I$. Then, $|\tilde I \cup \{i\}| = \ell + 1 \leq k$, and since $f$ is $k$-independent (and thus $(\ell + 1)$-independent), there exist two fixed points $x, y \in \FP(f)$  such that:
\begin{align*}
x_i = 0, y_i = 1, x_{\tilde I} = a = y_{\tilde I}
\end{align*}
Therefore, \(f_i(x) = f_i(y) = b\), but \(f_i(x) = x_i = 0\) and \(f_i(y) = y_i = 1\), which is a contradiction.
\end{proof}

\begin{corollary}
    If $G$ is a loopless $k$-admissible digraph, then its minimum indegree is at least $k$.
\end{corollary}

\begin{corollary}
      There is no AND-OR Boolean network $f$ with $i(f) \geq 2$ and loopless interaction graph.
\end{corollary}

\begin{remark}
     It is worth mentioning that the hypothesis of having no loops is necessary to conclude the previous results. For instance, consider the network $f: \{0,1\}^{n} \to \{0,1\}^{n}$ defined by $f_i(x) = x_i$, for $i = 1,\ldots, n-1$; and 

$$f_{n}(x) = x_{n}\lor \left(\bigwedge_{i = 1}^{n-1} \overline{x_i}\right).$$
Then, $G(f)$ has loops and $IC(f_i) = 1$ for every $i = 1,\ldots,n$. However, the set of fixed points of $f$ is $\{0,1\}^{n}\setminus\{\Vec{0}\}$, and this set is a covering array of strength $n-1$.
\end{remark}

\begin{remark}\label{remark: contraejemplo}
     As we have seen before, it is known that for $n\geq 4$, $CAN(n;2)>4$. On the other hand, the bound $CAN(n;k)\geq 2^{k_0}CAN(n-k_0; k-k_0)$, for $k_0\leq k$, is also known ~\cite{LJKRLYKRFM2011}. Using $k-k_0 = 2$, we can conclude that $CAN(n;k)>2^k$ for all $n>k+1$. This allows us to see that for all $k>1$, the conditions $\tau(G)\geq k$, $\delta^-(G)\geq k$, and that $G$ has no loops are necessary but not sufficient. Consider $n = k^2+k >k+1$, and a complete bipartite graph $G$, with one set of size $k$ and the other of size $k^2$. For this graph, $\tau(G) = k$ and $\delta^-(G) = k$. However, since $CAN(n;k)>k$, $G$ is not $k$-admissible.

\end{remark}

\begin{figure}
\centering
  \includegraphics[scale = 0.9]{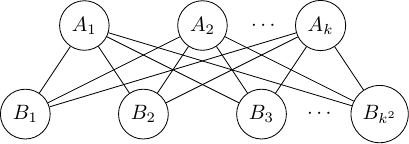}
\caption{Construction from Remark \ref{remark: contraejemplo}.}
  \label{bipartite complete k k2}

\end{figure}

\subsection{Families of \emph{k}-admissible graphs}
We have already reviewed some necessary conditions for $k$-admissibility in terms of the interaction graph and its local activation functions. On the other hand, from Proposition \ref{completegraph}, we observed that the complete graph is a suitable architecture for achieving high degrees of $k$-admissibility when considering linear networks. In this section, we will present two explicit constructions of $k$-admissible graphs for different values of $k$, inspired by the $(n-1)$-admissibility of the complete graph without loops.

\begin{proposition}\label{lema:pegarcliques}
    Let $r, s$ be two integers and define $\xi := \min\{r,s\}-1$. Then, there exists a $\xi$-admissible connected digraph on $n = r+s$ vertices.
\end{proposition}
\begin{proof}
    Let $K_r$ and $K_s$ denote the cliques on $r$ and $s$ vertices, respectively. Now we define $G$ composed by these two cliques and select $i\in V(K_r)$, and add all the arcs of the form $(i, \ell)$ for $\ell \in K_s$. Let $f:\{0,1\}^n\to\{0,1\}^n$ be a linear Boolean network with $G(f) = G$. Now, we see that for every $x\in \FP(f)$, if $x_i = 0$ the number of ones in both cliques should be even. So there are $2^{r-2}2^{s-1}$ fixed points. On the other case, if $x_i = 1$, every vector with an odd number of ones on the variables given by $K_s$, and an odd number of ones in $K_r\setminus \{i\}$, is a fixed point of $f$. In this case there are also $2^{r-2}2^{s-1}$ options. In total, there are $2^{r+s-2}$ fixed points and by previous lemmas this set is a covering array of strength $\xi$.
\end{proof}

\begin{figure}[H]
\centering
  \includegraphics[scale = 0.8]{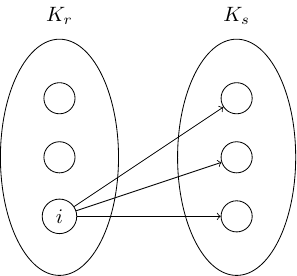}
\caption{Construction from Proposition \ref{lema:pegarcliques}.}

\end{figure}

It is worth mentioning that the previous construction only allows us to construct digraphs that are, at most, $n/2$-admissible. Moreover it provides examples of $\xi$-admissible graphs $G$ with $n$ vertices, $\tau(G)=n-2$ and $\delta^-(G) = \xi$. In the following construction, we generalize this result and show a family of strongly connected $k$-admissible graphs.

 Furthermore, we are particularly interested in finding, for a given $n$ and $k$, a family of strongly connected and $k$-admissible graphs. We were able to address this question for certain cases of $n$ and $k$ with the following lemma.

\begin{proposition}\label{windmillsxor}
    For any integer $m \geq 2$ and odd $k \geq 1$, there is a strongly connected graph that is $(m-1)$-admissible, with $n = (m-1)k + 1$ vertices.
\end{proposition}
\begin{proof}
    We know that cliques achieve high $k$-independence with linear
    functions. Our next construction is built upon this idea. Let $W_{m,k} = (V, E)$ be a graph with $n = (m-1)k+1$ vertices, comprising a central vertex and $k$ copies of $K_m$, each sharing only the central vertex. Examples of these graphs are shown in Figure \ref{windmills}.

We claim that for every $m,k$ with odd $k$, the linear Boolean network with interaction graph $W_{m,k}$ is $(m-1)$-independent. To prove this, we will first characterize the set of fixed points of this network. To do so, we denote by $f$ the linear BN with $G(f)=  W_{m,k}$, by $1$ the central vertex of the graph, and let $x\in \FP(f)$. Now, we distinguish the following two cases:
\begin{itemize}
    \item If $x_1 = 0$, then we need that the central vertex observes an even number of ones. 
    \item If $x_1 = 1$, then we need for it to observe an odd number of ones.
\end{itemize}
On the other hand, each of the cliques of size $m$ must have an even number of ones; otherwise, the configuration would be unstable. We denote by $K_{m-1}^1,\ldots, K_{m-1}^k$. Then, the set of fixed points of $f$ is given by the configurations that have $x_1=0$ and for every $\ell\in\{1,\ldots, k\}$, $w_H(x_{K_{m-1}}^\ell)$ is even
or $x_1 = 1$ and for every $\ell\in\{1,\ldots, k\}$, $w_H(x_{K_{m-1}}^\ell)$ is odd. Here we note that if $k$ is even, the central vertex cannot take the value 1 on a fixed point, because it will always observe an even number of ones. We can summarize the set of fixed points in the following table:

\begin{table}
\centering
\begin{tabular}{||c c c c c||} 
 \hline
 $1$ & $K_{m-1}^1$ &$K_{m-1}^2$  & $\cdots$ &$K_{m-1}^k$\\ [0.5ex] 
 \hline\hline
 0 & even & even & even & even\\
 $\vdots$ & $\vdots$ &  $\vdots$ &  $\vdots$ &  $\vdots$\\
 0 & even & even & even & even\\
 1 & odd & odd & odd & odd\\
 $\vdots$ & $\vdots$ &  $\vdots$ &  $\vdots$ &  $\vdots$\\
 1& odd & odd & odd & odd\\ [1ex] 
 \hline
\end{tabular}
\caption{Fixed points of $W_{m,k}$ with linear interaction, $k$ odd.}
\label{table:fxwmk}
\end{table}
Considering that for each $K_{m-1}^{\ell}$ there are $2^{m-2}$ possible configurations with even (or odd) weight, we have $2^{(m-2)k}$ fixed points with $x_1 = 0$ and the same amount with $x_1 = 1$. Thus, $f$ has $2^{(m-2)k+1}$ fixed points. Moreover, this set has strength $m-1$. Indeed, let $I$ be a subset of $m-1$ vertices from $W_{m,k}$ and let $a = (a_1, a^{K_{m-1}^{i_1}},\ldots, a^{K_{m-1}^{i_t}})\in\{0,1\}^{m-1}$, with $t\leq k$. We know that  the set of fixed points, for $x_1 = 0$ (or $x_1 = 1$) restricted to any $K_{m-1}^{\ell}$ is a covering array of strength $m-2$. Then, there exists a fixed point $x$ such that $x_I = a$, so $\FP(f)$ is a covering array of strength $m-1$.
\end{proof}

\begin{figure}
\centering
  \includegraphics[scale = 0.65]{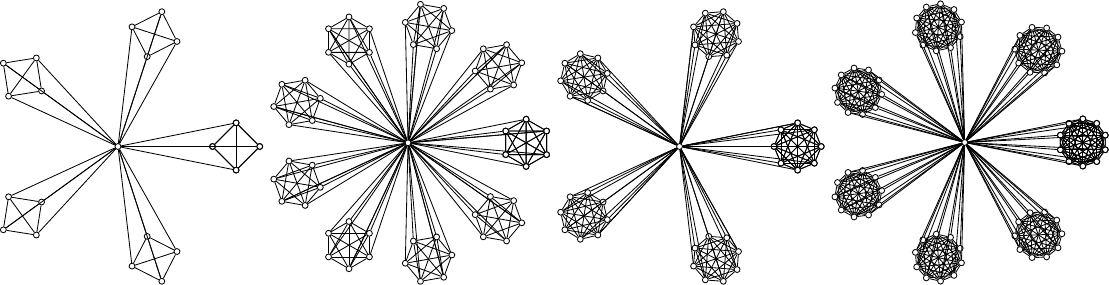}
\caption{Windmill graphs with $(m,k) \in\{(5,5),(7,9),(9, 5),(11,7)\}$ (left to right).}
  \label{windmills}

\end{figure}

\subsection{Constructions}\label{construcciones}
From the results of the previous section, we can observe a trade-off between the parameters $m$, $n$, and $k$ in an element of $CA(m,n;k)$. We aim to understand how to grow one of these parameters in terms of another, focusing on the context of a $k$-independent Boolean network on $n$ variables, with $m$ fixed points and $i(f) = k$. In addition, we will translate these results into constructions of $k$-admissible graphs.

The following result allows us to increase $n$ by one while maintaining strength in a certain sense.

\begin{lemma}[See e.g. ~\cite{LJKRLYKRFM2011}]\label{lema:incrementarn}
    Let $A\in CA(m_1, n-1; k)$ and $B\in CA(m_2, n-1;k-1)$. Then, 
    \begin{align*}
        C = \begin{bmatrix}
            A&\Vec{0}\\
            B&\Vec{1}
        \end{bmatrix} \in CA(m_1+m_2, n; k).
    \end{align*}
\end{lemma}
\begin{proof}
    Let $I = \{i_1,\ldots, i_k\}\subseteq [n]$ and $a = (a_1,\ldots, a_k)\in\{0,1\}^k$. Now there are two possible cases. If $n\not\in I$ since $A$ is a covering array of strength $k$, there is a vector $x\in C$ such that $x_I = a$. In the other case $n\in I$, and we write without loss of generality $I = \{i_1,\ldots, i_{k-1}, n\}$ and $a = (a_{i_1}, \ldots, a_{i_{k-1}}, a_n)$. If $a_n = 0$, since $A$ has strength $k$ there exists $x\in C$ such that $x_I = a$. Otherwise, if $a_n = 1$, as $B$ has strength $k-1$, there is a vector $y\in C$ such that $y_{I\setminus \{n\}}  = (a_{i_1},\ldots, a_{i_{k-1}})$, and therefore $y_I = a$.
\end{proof}

\begin{remark}\label{remark:incrementarn}
In the previous lemma, if we also assume $B \not\in CA(m_2, n-1; k)$, then $C$ is not an element of $CA(m_1 + m_2, n; k+1)$. Indeed, let $I \subseteq [n-1]$ and $a \in \{0,1\}^k$ be such that there is no $ x \in B $ with $x_I = a $. Consider $ \tilde{I} = I \cup \{n\}$ and $ \tilde{a} \in \{0,1\}^{k+1} $ such that $\tilde{a}_I = a $ and $ a_n = 1 $. Then, there is no $x \in C $ with $ x_{\tilde{I}} = \tilde{a} $, and therefore, $C$ does not have strength $k+1$.
\end{remark}

In terms of $k$-independent networks, Lemma \ref{lema:incrementarn} and Remark \ref{remark:incrementarn} allows us to establish the following result.

\begin{proposition}\label{prop: unirgrafos}
    Let $f$ be a Boolean network on $n-1$ variables with $i(f) = k-1$. Then, there exists a Boolean network $g$ on $n$ variables, with $i(g) = k$ and $G(g)$ connected such that $\FP(g)_{[n-1]} :=\{(x_1,\ldots, x_{n-1})\in\{0,1\}^{n-1}\,:\, (x_1,\ldots, x_{n-1}, x_n)\in \FP(g)\}$ contains the set of fixed points of $f$.
\end{proposition}
\begin{proof}
    Let $\tilde f:\{0,1\}^{n-1}\to\{0,1\}^{n-1}$ such that $i(f) = k$ (exists by Proposition \ref{completegraph}).  Now, define 
    \begin{align*}
g_i(x) = (x_{n} \land  f_i(x)) \lor (\overline{x_{n}} \land \tilde f_i(x)), \,\,\, i\in\{1,\ldots,n-1\},
\end{align*}
and $g_{n}(x) = x_{n}$. Note that if $x_{n} = 0$, then $g(x) = \tilde f(x)$, while if $x_{n} = 1$, $g(x) =  f(x)$. So, the set of fixed points of $g$ is
    \begin{align*}
        \FP(g) = \begin{bmatrix}
            \FP(\tilde f) & \Vec{0}\\\FP(f)&\Vec{1}
        \end{bmatrix}
    \end{align*}
    And by Lemma \ref{lema:incrementarn} and Remark \ref{remark:incrementarn}, $\FP(g)\in CA(n; k)\setminus CAN(n; k+1)$ and therefore $i(g) =  k$. Moreover, if we suppose $\FP(f)$ and $\FP(\tilde f)$ are disjoint we can avoid the loop in $n$ by repeating the previous argument with $g_{n}(x)$ as the indicator function of $\FP( f)$, i.e., $g_n(x) = 1$ if $x\in \FP(f)$ and $g_n(x) = 0$ if $x\in \FP(\tilde f)$.
\end{proof}

\begin{remark}
    We can also state Proposition \ref{prop: unirgrafos} in the following manner: Given $G^1, G^2$ to graphs on $V = [n]$, such that $G^1$ is $k$-admissible and $G^2$ is $(k-1)$-admissible, then we can construct $\tilde G = (\tilde V,\tilde E)$, where $\tilde V = [n+1]$ and $\tilde E = E(G^1)\cup E(G^2)$. Thus, by the previous proposition, we can define the same network and conclude that $\tilde G$ is a $k$-admissible graph on $n+1$ vertices. In Figure \ref{construccionunion}, we observe an example of this construction considering the Maj network in $G^1$, being $2$-independent, and the linear network in $G^2$ achieving $3$-independence. In this case, $\tilde G$ is the resulting graph, which turns out to be $3$-admissible with the network defined in Proposition \ref{prop: unirgrafos}.
\end{remark}

\begin{figure}
\centering
  \includegraphics[scale = 0.4]{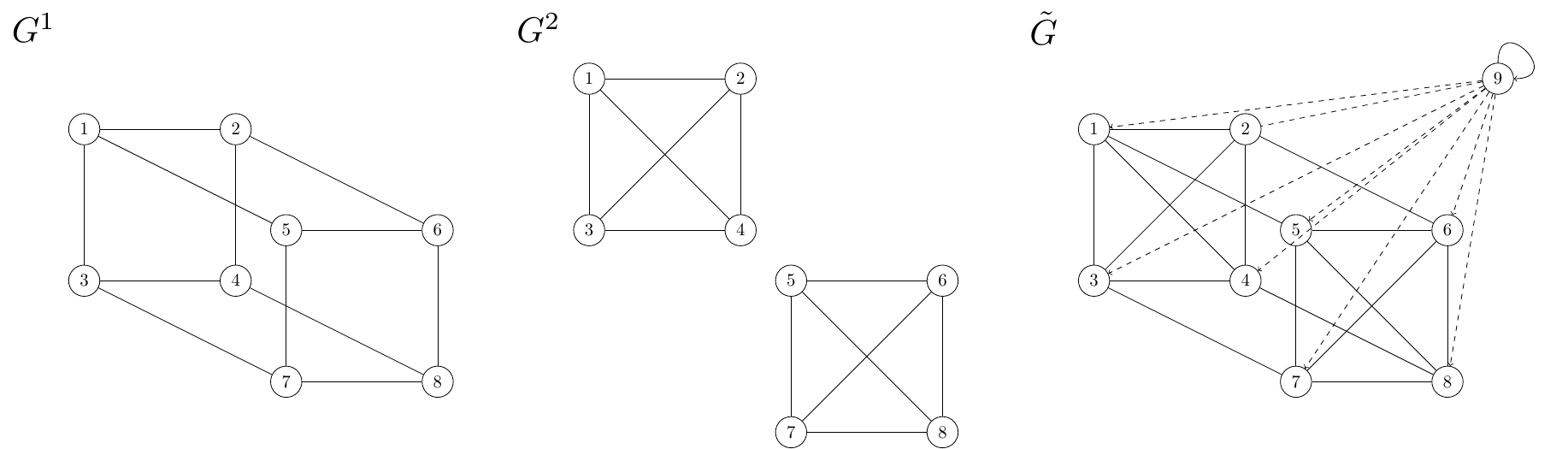}
\caption{Construction from Proposition \ref{prop: unirgrafos} using $G^1$ with majority and $G^2$ with linear functions.}
  \label{construccionunion}

\end{figure}
The following remark shows that by adding an isolated loop, we can increase $n$ by one while maintaining the strength. This, in turn, implies doubling the value of $m$, i.e., the number of fixed points.

\begin{remark}\label{prop:increasevertex}
    Given a $k$-admissible graph on $n$ vertices, $G$, the addition of an isolated loop would return a $k$-admissible graph on $n+1$ vertices. Indeed, let $f$ be a $k$-independent BN with interaction graph $G$. Now we define $\tilde f:\{0,1\}^{n+1}\to\{0,1\}^{n+1}$ as $\tilde f(x) = (f_1(x),\ldots, f_n(x), x_{n+1})$. So $G(\tilde f) = \tilde G$, and also
    \begin{align*}
        \FP(\tilde f) = \begin{bmatrix}
            \FP(f) & \Vec{0}\\
            \FP(f) &\Vec{1}
        \end{bmatrix}
    \end{align*}
    Now, by Lemma \ref{lema:incrementarn},     $\FP(\tilde f)\in CA(2|\FP(f)|, n+1; k)$. This construction also allows us to use cliques with linear functions and isolated loops to construct, for any $n$ and $k$, Boolean networks with $i(f) = k$, and non-complete interaction graph. Additionally, if $n$ is a multiple of $k$, incorporating disjoint copies of cliques of size $k$ into this construction results in a $(k-1)$-regular, $(k-1)$-admissible graph on $n$ vertices.

\end{remark}

After recognizing that the inclusion of loops doubles the number of fixed points, we wonder: Can we construct examples of networks with $i(f) = k$ and the maximum number of fixed points without increasing the strength? To advance in this direction, we first prove the following upper bound.

\begin{proposition}\label{upperbound}
    Let $A\in CA(n;k)\setminus CA(n;k+1)$. Then, an upper bound for the number of elements of $A$ is
    \begin{align*}
        2^{n-1}(2-2^{-k})
    \end{align*}
\end{proposition}

\begin{proof}
    Since $A$ has no strength $k+1$, there exists $a = (a_1,\ldots, a_{k+1})\in \{0,1\}^{k+1}$ such that for any vector we select as a completion $b = (b_{k+2}, \ldots, b_n)\in \{0,1\}^{n-k-1}$, the concatenation $ab = (a_1,\ldots, a_{k+1},b_{k+2},\ldots, b_n)\in\{0,1\}^n$ is not an element of $A$. Therefore, there are at least $2^{n-k-1}$ elements that are not part of the rows of $A$, so the upper bound is $2^n-2^{n-k-1} = 2^{n-1}(2-2^{-k})$.
\end{proof}

Now consider a graph $G$ composed by a clique of size $k+1$ and $n-k-1$ isolated loops. Suppose we have a linear Boolean network with this interaction graph. Then, by the previous results, we know that $i(f) = k$. The inclusion of loops does not increase the strength, as the configuration $\Vec{1} \in \{0,1\}^{k+1}$ remains unstable for the isolated clique. Then, since every loop duplicates the set of fixed points, we conclude that $f$ has $2^{n-k-1}2^{k} = 2^{n-1}$ fixed points. This result demonstrates that, for a fixed strength $k$, we can approach the bound from Proposition \ref{upperbound} closely (up to a constant in terms of $k$)

\begin{corollary}\label{prop: maximumnumbervectors}
       For every $k\leq n$, there is a Boolean network with $i(f) = k$ and $2^{n-1}$ fixed points.
\end{corollary}

\begin{figure}[H]
\centering
  \includegraphics[scale = 0.8]{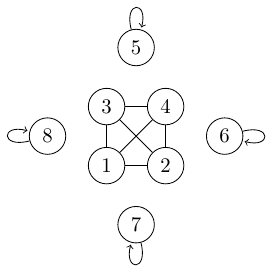}
\caption{Construction from Proposition \ref{prop: maximumnumbervectors} with $n = 8$ and $k = 3$.}

\end{figure}

By using a different approach, the following result allows us to significantly increase $n$ while keeping the strength controlled.

\begin{lemma}\label{Lema:directproduct}
    Let $A \in CA(m_s, n_s; s)$ and $B\in CA(m_r, n_r; r)$. We denote by $A\otimes B$ the set of all possible concatenations between a vector of $A$ and a vector of $B$:
        \begin{align*}
            A\otimes B = \{a_ib_j\in \{0,1\}^{n_s+n_r}\,:\, i,j\in [s]\times [r]\}.
        \end{align*}
    Then, $A\otimes B\in CA(m_sm_r, n_s+n_r; t)$, where $t = \min\{r,s\}$.
\end{lemma}
\begin{proof}
    Without loss of generality, assume $t = s$. Let $I = \{i_1,\ldots, i_s\}\subseteq [n_s+n_r]$. Consider the partition of $I$ into $I_A$ and $I_B$, where $I_A$ contains the $\ell_A$ indices between $1$ and $n_s$, and $I_B$ contains the $\ell_B$ indices between $n_s+1$ and $n_s + n_r$. Let $a = a^Aa^B\in\{0,1\}^{n_s+n_r}$, where $a^A = (a^A_1,\ldots, a^A_{\ell_A})$ and $a^B = (a^B_1,\ldots, a^B_{\ell_r})$. Since $t = \min\{s,t\}$, we know that $A$ and $B$ are covering arrays of strength $s$. Thus, there exist $x\in A$ and $y\in B$ such that $x|_A = a^A$ and $y|_B = a^B$. As $A\otimes B$ contains all possible concatenations of elements between $A$ and $B$, we conclude that $xy\in A\otimes B$ and, therefore, $A\otimes B\in CA(m_sm_r, n_s+n_r; t)$.
\end{proof}
\begin{corollary}\label{corolario:directproduct}
    Let $\{A^\ell\}_{\ell = 1}^L$ be a collection of sets of Boolean vectors such that for every $\ell$,  $A^\ell$ is an element of $CA(m_\ell, n_\ell; t_\ell)$. Then, 
    $$\bigotimes_{\ell = 1}^L A^\ell = ((A^1\otimes A^2)\otimes A^3)\otimes \cdots \otimes A^L)\in CA(m, n;t)$$
    Where $m = \prod_{\ell = 1}^Lm_\ell$, $n_\ell = \sum_{\ell = 1}^L n_\ell$ and $t = \min\{t_\ell\, :\, \ell = 1,\ldots, L\}$.
\end{corollary}

\begin{remark}
Consider a family of Boolean networks $\{f_\ell\}_{\ell = 1}^L$ such that for each $\ell$, $G(f^\ell) = G_\ell$ and $i(f^\ell) = t_\ell$. Define the graph $G = \bigcup_{\ell = 1}^L G_\ell$ by
\begin{equation*}
    V(G) = \bigcup_{\ell = 1}^L V(G_\ell),\quad E(G) = \bigcup_{\ell = 1}^L E(G^\ell).
\end{equation*}Then, there exists a Boolean network $f$ such that $G(f) = G$ and $i(f) = k$, where $k = \min\{t_\ell \,:\, \ell = 1,\ldots, L\}$. Indeed, since $G$ is a disjoint union, we can define $f$ locally as $f^\ell$ for each $G_\ell$. Thus, the set of fixed points is of the form:
\begin{align*}
\FP(f) = \bigotimes_{\ell = 1}^L \FP(f^\ell)
\end{align*}
where each $\FP(f^\ell)$ is a covering array of strength $t_\ell$. Then, by Lemma \ref{Lema:directproduct}, this set is a covering array of strength $k = \min\{t_\ell \,:\, \ell = 1,\ldots, L\}$ with $\prod_{\ell = 1}^L |\FP(f^\ell)|$ elements.
\end{remark}

The preceding Remark shows that we can use Corollary \ref{corolario:directproduct} to, from a family of networks with certain degrees of $k$-independence, construct another one (increasing $n$ and $m$, and controlling $k$), with a disconnected interaction graph. The following result demonstrates that we can also carry out a similar construction, but while maintaining the interaction graph strongly connected.

\begin{proposition}\label{fuerteconexo1}
Let $\{f_\ell\}_{\ell = 1}^L$ be a family of Boolean networks such that for each $\ell$, $G(f^\ell) = G_\ell$ and $i(f^\ell) =  t_\ell$. Then, there is Boolean network $f$ with a strongly  connected interaction graph $G(f) = G$  and $i(f) = k$, where $k= \min\{t_\ell\,:\, \ell = 1,\ldots, L\}$.
\end{proposition}
\begin{proof}
    Define $G$ with vertex set $V:= \bigcup_{\ell = 1}^L V(G_\ell)$, $n = |V|$ and consider a Boolean network $f:\{0,1\}^n\to\{0,1\}^n$ such that for every $i\in G_1$, $f_i$ is defined by
    \begin{align*}
        f_i(x) = f_i^1(x_{G_1})\land C_2(x_{G_2})\land \cdots \land C_L(x_{G_L}),
    \end{align*}
    where $C_\ell(x) = 1$ if and only if $x_{G_\ell}\in \FP(f^\ell)$. We also define for every $\ell\in\{2,\ldots, L\}$, and for every $j\in G_\ell$,
    \begin{align*}
    f_j(x) = f_j^{\ell}(x_{G_\ell}) \land C_1(x_{G_1})
    \end{align*}
    Then, it is easy to see that 
    \begin{align*}
        \FP(f) = \bigotimes_{\ell = 1}^L \FP(f^\ell) \in CA(n; k).
    \end{align*}
  Finally, recall that we assume $i(f^\ell) = t_\ell$ for every $\ell$. Suppose, for contradiction, that $i(f) = k+1$. Consider $I = \{i_1, \ldots, i_{k+1}\} \subseteq V(G_\ell)$. For every $a \in \{0,1\}^{k+1}$, there would exist $x \in \FP(f^\ell)$ such that $x_I = a$, implying $i(f^\ell) \geq k+1 > t_\ell$, which contradicts our assumption.
\end{proof}

\section{The monotone case}\label{seccionmonotonas}

As we saw in Remark \ref{no son monotonas}, the general construction of Boolean networks with $n$ variables and $i(f) = k$ does not guarantee the existence of monotone $k$-independent networks. Similarly, the other constructions presented in the previous chapter do not provide results on the existence of $k$-admissible graphs with monotone networks. There have been previous studies on fixed points in monotone networks, but they do not consider the structure of the set of fixed points ~\cite{aracena2017number}. This theoretically motivates us to question whether monotone networks can be $k$-independent for some $1 < k < n$. Additionally, this question is interesting from an applied perspective, as networks modeling binary opinion exchange systems are often monotone. Therefore, we dedicate this section to studying the relationship between monotonicity and $k$-independence.

The following combinatorial design proves to be convenient when working with covering arrays and monotone networks.
\begin{definition}
     Let $A = \{x^1,\ldots, x^ m\}\subseteq \{0,1\}^n$. We say that $A$ is a Steiner system with parameters $(n, k, t)$ if $w_H(x^i) = k$ for $i 
 = 1,\ldots, m$, and for every subset of indices $I = \{i_1,\ldots, i_t\}$ there is an unique vector $x^j\in A$ such that $x^j_{i_\ell} = 1$ for $\ell\in\{1,\ldots, t\}$.
\end{definition}

   Given a set of indices $I = \{i_1, \ldots, i_t\}$ and values $a = (a_1, \ldots, a_t) \in \{0,1\}^t$, we say that a vector $x \in \{0,1\}^n$ such that $x_I = a$ is a completion of $a$. In this context, a Steiner system guarantees the uniqueness of the completion of the configuration $\Vec{1} \in \{0,1\}^t$ for any subset of $t$ indices.

As an example, the following is a Steiner system with parameters $(8,4,3)$:
    \begin{align*}
        A = \begin{matrix}
            11010001\\
            01101001\\
            00110101\\
00011011\\
10001101\\
01000111\\
10100011\\
00101110\\
10010110\\
11001010\\
11100100\\
01110010\\
10111000\\
01011100\\
        \end{matrix}
    \end{align*}

The existence of a Steiner system with given parameters has been a fundamental problem in combinatorics ~\cite{doyen1980updated}. In a broader context, divisibility conditions were established: for an $(n, q, r)$ Steiner system to exist, a necessary condition is that $\binom{q-i}{r-i}$ divides $\binom{n-i}{r-i}$ for every $0 \leq i \leq r-1$. For many years, it was conjectured that these divisibility conditions were also sufficient. This conjecture was proven in 2014 for large values of $n$ ~\cite{keevash2014existence}. See also ~\cite{glock2016existence}, ~\cite{barber2020minimalist}.

\begin{lemma}
Let $A$ be a Steiner system with parameters $(n, t+1,t)$ such that $2t<n$. Then, $A\in CA(n;t)\setminus CA(n;t+1)$.
\end{lemma}

\begin{proof}
    Let $I$ be a subset of $[n]$ of size $t$, we will assume without loss of generality that $I = \{1,\ldots, t\}$. We aim to prove that for every $a = (a_1,\ldots, a_t)\in\{0,1\}^t$, there exists  $x\in A$ such that $x_I = a$. We will proceed with the proof by induction on the number of zeros in $a$.

First, observe that there exists $x^\ell\in A$ such that $x^\ell_I = 11\cdots 1$, due to the property of Steiner systems. As the vectors have weight $t+1$, there exists a unique $w\in\{t+1,\ldots, n\}$ such that $x^\ell_{w} = 1$.
Let $i_0 \in I$ and let $\overline{e_{i_0}} \in \{0,1\}^t$ be the vector that has a single zero at position $i_0$ and define $K^{i_0} =  \left(\{1,\ldots, t\}\setminus\{i_0\}\right)\cup\{w\}$. Notice  $K^{i_0}$ is a subset of $t$ indices, so there exists a vector $x^{\ell_1}\in A$ that has ones in the components $K^{i_0}$. Suppose $x^{\ell_1}_{i_0} = 1$. In such case, $x^\ell$ and $x^{\ell_1}$ would be two vectors in $A$ that has ones in $I$, which contradicts the definition of a Steiner system.
 Therefore, $x^{\ell_1}_{i_0}$ must be zero, and hence $x^{\ell_1}_I = \overline{e_{i_0}}$. With this, we proved that given a subset of $t$ indices, all configurations with one zero and $t-1$ ones appear.

Now, suppose that all configurations with $s$ zeros appear, and let us prove that those with $s+1$ zeros also appear. Let $a = (a_1,\ldots, a_t)\in \{0,1\}^t$ such that $a_1= \cdots = a_{s+1} = 0$ and $a_{s+2} = \cdots a_t = 1$. We will prove that there exists an element of the Steiner system that takes the values of $a$ at the indices $I$. Consider the vector $x^{s}$ that completes the configuration $b = (b_1,\ldots,b_t)$ with values $b_1 = \cdots = b_{s} = 0$, $b_{s+1} = \cdots = b_t = 1$ (which exists by the induction hypothesis).
Now, let $J = \{\ell\in{t+1,\ldots,n} \, :\, x^s_\ell = 1\}$. As the vectors of the Steiner system have weight $t+1$, $|J| = s+1$. We denote $J = \{j_1,\ldots, j_s, j_{s+1}\}$, and consider $w\in \{t+1, \ldots, n\}\setminus J$, which allows us to define $K^i =  (\{j_1,\ldots, j_s\}\cup\{w\})\cup\{s+2, \ldots, t\}$, which is a subset of $[n]$ of size $t$, so there exists $y\in A$ that takes the value one in the components indexed by $K^i$, and also has another component with value one.  Note that if $y_{s+1} = 1$, we would have two different completions for $\{s+1,\ldots, t\}\cup J\setminus\{j_{s+1}\}$, which is a contradiction. Now, if there exists $\ell\in\{1,\ldots, s\}$ such that $y_\ell = 1$, we can consider, instead of $x^s$, the vector $\xi^s$ such that $\xi^s_\ell = 1$, $\xi^s_{s+2} = \cdots= \xi^s_t = 1$, and define $J$ based on $\xi^s$, and thus repeat the same argument as before. We thus conclude that there must exist $\zeta \in\{t+1,\ldots, n\}\setminus K^i$ such that $y_\zeta = 1$, and therefore $y_I = a$.

Finally, it is easy to see that $A$ cannot be a covering array of strength $t+1$. Indeed, suppose it is, and let $I = \{1,\ldots, t+1\}$. The existence of a configuration $x$ that has all its ones in $I$ and a vector $y$ that has $t$ ones in $I$ implies two different completions for the vector of ones in $\{j\in I: x_j = y_j = 1\}$, leading to a contradiction.
\end{proof}

\begin{theorem}\label{existenciamonotonas}
    Given a Steiner system $A$ with parameters $(n, t+1, t)$, where $2 \leq t < n/2$, there exists a monotone Boolean network $f$ such that $i(f) = t$ and $G(f) = K_n$, with fixed points that include $A$.
\end{theorem}
\begin{proof}
    Let $A = \{y^1,\ldots, y^m\}$ be a $(n, t+1, t)$-Steiner system. By the previous lemma, we know that $A$ is a covering array of strength $t$. Now for every $i\in [n]$ we define the Boolean function
    \begin{align*} 
        f_i(x_1,\ldots, x_n) = \bigvee_{\{k\;:\,y^k_i = 1\}} \bigwedge_{\{j\neq i\,:\, y^k_j = 1\}} x_j.
    \end{align*}
Now we will prove that $A\cup \{\Vec{0}, \Vec{1}\} \subseteq \FP (f)$. Indeed, it is clear that $\Vec{0}$ and $ \Vec{1}$ are fixed points of $f$. Let $y^\ell\in A$, and let us prove that $f(y^\ell) = y^\ell$. Let $i\in[n]$, and suppose initially that $y^\ell_i = 0$. By contradiction, suppose $f_i(y^\ell) = 1$, and therefore there exists $k\in [m]$ where $y^k_i = 1$ and for every $j\neq i$ such that $y_j^k = 1$, we have that $y_j^\ell = 1$.  Notice that the above would imply that the index set $I = \{j\neq i\,:\, y_j^k = 1\}$, which has size $t$, has two different completions, one by $y^\ell$ and the other by $y^k$. This contradicts the uniqueness of the definition of Steiner systems. On the other hand, suppose now that $y_i^\ell = 1$. In this case, within the expression for $f_i(y^\ell)$, the following conjunction appears: $$ \bigwedge_{\{j\neq i\,:\, y_j^{\ell}=1\}} y_j^\ell$$
Therefore, $f_i(y^\ell) = 1$. 
This implies that for any $y^\ell$ in $A$, $f(y^\ell) = y^\ell $, which is equivalent to $A \subseteq \FP(f)$, and therefore $i(f) \geq t $. Moreover, by definition $IC(f_i) = t$ for every $i \in [n]$. Using the contrapositive of Theorem \ref{teoremanecesario}, we can conclude that $ i(f) < t+1 $, and thus $i(f) = t $.

Now we will prove that $G(f) = K_n$. To do this, we first notice that since $f_i$ can be written as a DNF formula without negated variables, $f_i$ is monotone and it depends on the variable $x_j$ if it appears in any clause. That is, $(j,i)$ is an arc in $G(f)$ if and only if there exists $y^k \in A$ such that $y_i^k = 1$ and $y_j^k = 1$, with $j \neq i$. Indeed, if $i \neq j \in [n]$, then we can consider any completion $T \subseteq [n]\setminus\{i,j\}$ with $|T| = t-2$. Then, by considering $T\cup\{i,j\}$, we have a subset of $t$ indices in $[n]$, and by definition, there exists a unique $y^k \in A\subseteq \FP(f)$ such that $y^k_i = y^k_j$ and $y_T = \Vec{1}$. Therefore, $(j,i) \in G(f)$, and as these are two arbitrary vertices, we conclude that $G(f) = K_n$.
\end{proof}

For example, the set $$A = \begin{matrix}
1101000\\0110100\\ 0011010\\ 0001101\\ 1000110\\ 0100011\\ 1010001\end{matrix}$$ is a Steiner system with parameters $(7,3,2)$. The previous construction gives us the 2-independent monotone network
\begin{align*}
    f_1(x) &= (x_2\land x_4)\lor (x_5\land x_6) \lor (x_3\land x_7)\\
    f_2(x) &=(x_1\land x_4)\lor (x_3\land x_5) \lor (x_6\land x_7)\\
    f_3(x) &=(x_2\land x_5)\lor (x_4\land x_6) \lor (x_1\land x_7)\\
    f_4(x) &=(x_1\land x_2)\lor (x_3\land x_6) \lor (x_5\land x_7)\\
    f_5(x) &=(x_2\land x_3)\lor (x_4\land x_7) \lor (x_1\land x_6)\\
    f_6(x) &=(x_3\land x_4)\lor (x_1\land x_5) \lor (x_2\land x_7)\\
    f_7(x) &=(x_4\land x_5)\lor (x_2\land x_6) \lor (x_1\land x_3)
\end{align*}
 %

Finally, we conclude this section by showing that monotone networks on $n$ variables cannot achieve independence number greater than $n/2$. We will state a more general proposition for regulatory networks.

\begin{proposition}
    Let $h:\{0,1\}^n\to\{0,1\}$ be an unate Boolean function. Define $\gamma^+ := \{i\in[n]\,:\, h \text{ is increasing on } i\}$ and $\gamma^- := \{i\in[n]\,:\, h \text{ is decreasing on } i\}$. Now we define a weight function, $\tilde w$, such that for every $x\in \{0,1\}^n$,
    \begin{align*}
        \tilde w(x) := |\{i\in \gamma^+\,:\, x_i = 1 \}| + |\{j\in \gamma^-\,:\, x_i = 0 \}|.
    \end{align*}
    Then, 
    \begin{align*}
        \max\left\{ \max_{\{x\,:\, h(x) = 1\}} (n-\tilde w(x)), \max_{\{y\,:\, h(y) = 0\}}\tilde w(y) \right\}\geq n/2.
    \end{align*}
\end{proposition}
\begin{proof}
    Denote by $\xi(h)$ the maximum from the proposition above. Suppose by contradiction that $\xi(h)< n/2$ and assume that the maximum is attained in the second element. That is, there exists $y\in \{0,1\}^n$ such that $h(y) = 0$ and $\xi(h) = \tilde w(y)$. Then, for every $x\in\{0,1\}^n$ with $\tilde w(x)>\tilde w(y)$, we have $h(x) = 1$. In particular, there exists $z\in\{0,1\}^n$ such that $\tilde w(z) = \tilde w(y) + 1 = \xi(h) + 1$ and $h(z) = 1$. Therefore, 
    \begin{align*}
        \max_{\{x\,:\, h(x) = 1\}}(n-\tilde w(x)) \geq n-\tilde w(z)\geq n/2,
    \end{align*}
    which contradicts the assumption that $\xi(h)<n/2$. The proof in the case where the maximum is reached at the first element follows analogously.
\end{proof}
Before stating the following corollary, we will need to consider an alternative way of viewing $k$-set canalizing functions. To do so, recall that $\{0,1\}^n$ is the set of vertices of an $n$-cube $Q_n$, and that any Boolean function $h:\{0,1\}^n\to\{0,1\}$ can be understood as a coloring of the vertices of the $n$-cube with two colors (0 and 1). Now, fixing $k$ variables and considering all vectors that have these variables fixed translates into viewing a $(n-k)$-subcube of $Q_n$. Therefore, a function $h$ is $k$-set canalizing if and only if $Q_n$ has a monochromatic $Q_{n-k}$ according to the coloring given by $h$.
\begin{corollary}
    Let $h:\{0,1\}^n\to\{0,1\}$ be an unate function, then $IC(h)\leq n/2$.
\end{corollary}
\begin{proof}
     Consider $\gamma^+$ and $\gamma^-$ defined in the same manner than the previous proposition. Suppose first that $\xi(h)$ is attained in the second maximum and $y\in\{0,1\}^n$ satisfies $\tilde w(y ) = \xi(h)$. Denote $\gamma^+_1(x) =\{i\in \gamma^+\,:\, x_i = 1\}$ and $\gamma^+_0(x) =\{i\in \gamma^+\,:\, x_i = 0\}$ (and analogously $\gamma^-_0(x), \gamma^-_1(x)$) for $x\in\{0,1\}^n$ and consider 
     \begin{align*}
         S_y = \{x\in\{0,1\}^n\,:\,\gamma^+_0(x) = \gamma^+_0(y)\, \land\, \gamma_1^-(x) = \gamma_1^-(y)\}.
     \end{align*}
     Recall that $\tilde w(y) = \gamma_1^+(y)  + \gamma_0^-(y)$. Note that $S_y$ is a set of vectors in $\{0,1\}^n$ that originates from fixing $\gamma_0^+(y) + \gamma_1^-(y) = n -\tilde w(y)$ variables, and therefore, it is a $\tilde \xi(h)$-subcube of $Q_n$. Now observe that we are fixing all increasing variables that are zero and all decreasing variables that are one in $y$. Consider $x\in S_y$ with $x < y$; given that the free decreasing variables of $y$ are zero, it necessarily follows that $x_{\gamma^+}< y_{\gamma^+}$, and therefore $h(x) = 0$. On the other hand, now consider $x\in S_y$ such that $x > y$. Since the increasing variables not fixed in $y$ are all ones, it necessarily follows that $x_{\gamma^-}> y_{\gamma^-}$, and therefore $h(x) = 0$. For every $x\in S_y$ such that $w_H(x) = w_H(y)$ we are not able to determine if $h(x) = 0$ or not. However, we can consider \begin{align*}
         S_{<y} = \{x\in S_y\,:\, x< y\}\quad \text{or}\quad  S_{>y} = \{x\in S_y\,:\, x> y\},
     \end{align*}
    where both sets contain a zero monochromatic $Q_{\xi(h)-1}$, which implies $$IC(h)\leq n-\xi(h)+1\leq n/2.$$
    Finally, if $\xi(h)$ is attained in the first element, a similar argument can be developed by considering $y\in\{0,1\}^n$ such that $n-\tilde w(y) = \xi(h)$, $h(y) = 1$ and
    $$S_y = \{x\in\{0,1\}^n\,:\, \gamma_1^+(x) = \gamma_1^+(y)\,\land\, \gamma_0^-(x) = \gamma_0^-(y)\}.$$
\end{proof}

\begin{corollary}
    There is no $k$-independent monotone Boolean network with $k>n/2$.
\end{corollary}
\begin{proof}
     Suppose there exists a $k$-independent monotone Boolean network $f:\{0,1\}^n\to\{0,1\}^n$ with $k> n/2$. Theorem \ref{teoremanecesario}, for every $i\in [n]$, $IC(f_i)\geq k >n/2$. Since $f_i$ is monotone, it is unate, and by the previous result, $IC(f_i)\leq n/2$.
\end{proof}

\section{Concluding remarks and open problems}
We introduced the concept of $k$-independent Boolean networks and addressed fundamental questions about their existence, in the general case, through Theorem \ref{completegraph}, and in the monotone case, through Theorem \ref{existenciamonotonas}, for specific values of $n$ and $k$ determined by the existence of Steiner systems with those parameters. Furthermore, we derived necessary conditions in terms of the interaction graph to represent a $k$-independent network, as detailed in Theorem \ref{teoremanecesario} and its respective corollaries. On the other hand, we also presented constructions that demonstrate the existence of networks with fixed $i(f)$ and disconnected interaction graph, as shown in Remark \ref{prop:increasevertex}; with connected interaction digraph, as detailed in Proposition \ref{lema:pegarcliques}; with strongly connected graph, as presented in Proposition \ref{fuerteconexo1} and Proposition \ref{windmillsxor}. Additionally, we explored constructions showing how the parameters $m$, $n$, and $k$ vary for networks $f$ in $n$ variables, with $i(f) = k$ and $m$ fixed points, as described in Proposition \ref{prop: maximumnumbervectors}.

Furthermore, there is a wide range of open questions, such as the general existence of monotone Boolean networks in $n$ variables with $1\leq k <n$. Similarly to what was discussed in Section \ref{construcciones}, constructions are also needed to vary the parameters of monotone networks. Likewise, characterizations of networks with $i(f) = k$ in terms of structural properties of the interaction graph, for a specific family of networks, remain to be discovered. We believe it would be interesting to adapt and utilize results from coding theory to advance in this direction. Similarly, we believe it could be interesting to explore Boolean networks whose sets of fixed points exhibit other combinatorial structures, such as Orthogonal arrays ~\cite{liu1995survey}, Covering arrays avoiding Forbidden Edges ~\cite{danziger2009covering}, Covering arrays on graphs ~\cite{meagher2005covering}, or more generally, to investigate how parameters studied in set-systems (e.g., ~\cite{lovasz1968chromatic}) translate to the set of fixed points and understand their implications in terms of the interaction graph.





\end{document}